%% file: main.tex
\documentclass[reqno]{asl}

\usepackage{setspace}
\title{On the Concept of Arithmetic Consequence}
\author{}
\date{}



\author{Alexander V. Gheorghiu }
\revauthor{Gheorghiu, Alexander V.}

\address{\textbf{ORCID:} 0000-0002-7144-6910}

\address{School of Electronics and Computer Science, University of Southampton\\
University Road, Southampton, SO17 1BJ 
United Kingdom}

\address{Department of Computer Science, University College London\\
Gower St, London WC1E 6BT, UK}

\email{a.v.gheorghiu@soton.ac.uk}

\makeatletter
\def\namedlabel#1#2{\begingroup
   \def\@currentlabel{#2}%
   \label{#1}\endgroup
}
\makeatother

\input{macros}

\begin{document}

\begin{abstract}
    G\"odel’s second incompleteness theorem is standardly understood as showing that no sufficiently strong, consistent theory of arithmetic can prove its own consistency, a result typically interpreted against a model-theoretic background in which arithmetical language is evaluated with respect to an independently given structure of natural numbers. This paper develops an alternative perspective grounded in proof-theoretic semantics. We distinguish between derivability and a semantic notion of consequence given by support, defined compositionally in terms of the inferential roles fixed by a theory. For suitable arithmetical theories $\pa$ formulated in a finite signature  (such as Robinson's \textsf{Q} and Peano Arithemtic), these two notions can diverge in a principled way: although $\pa$ does not prove its own consistency, it nevertheless supports its formalized consistency statement, and more generally supports sentences not derivable within it. This does not conflict with G\"odel's incompleteness theorem, but instead reframes incompleteness as a divergence between two internally determined notions of consequence associated with a single theory, rather than as a gap between syntactic provability and truth in a mind-independent structure. The result clarifies the relationship between reflection, consistency, and inferentialist approaches to meaning, and shows how substantial semantic determinacy may arise from the inferential structure of arithmetic itself.
\end{abstract}

\maketitle

\section{Introduction} \label{sec:introduction}

G\"odel's incompleteness theorems~\cite{godel1931undecidable} establish that any recursively enumerable 
consistent formal system $\pa$ sufficient to represent elementary number theory must leave certain 
sentences neither provable nor refutable within it. The consistency statement $\Con(\pa)$ is a 
canonical example: $\not\vdash_{\pa} \Con(\pa)$ and $\not\vdash_{\pa} \neg\Con(\pa)$. These results 
are well-known, not least for their decisive bearing on Hilbert's programme: the 
second incompleteness theorem forecloses the possibility of securing the consistency of arithmetic by 
finitary means within the system itself. G\"{o}del's own philosophical remarks~\cite{godel1951} drew 
a further consequence: since $\Con(\pa)$ is true yet unprovable, mathematical truth cannot be 
identified with formal provability, and our formal systems must be understood as answering --- with 
varying degrees of completeness --- to an independent mathematical reality.

Tarski provided what has become the standard mathematical articulation of this independence. 
Observing in ``On the Concept of Logical Consequence''~\cite{tarski1936following_eng} that the 
intuitive notion of logical consequence remained without a precise definition despite --- indeed, 
partly because of --- G\"{o}del's results, he applied his theory of truth for formal 
languages~\cite{tarski1933truth_eng} to close the gap. The central move is to distinguish two notions 
of arithmetic consequence: \emph{syntactic consequence}, what is derivable from the axioms by formal 
proof, and \emph{semantic consequence}, what holds in every model satisfying those axioms. 

The semantics in question is what is now called \emph{model-theoretic semantics}. In this framework, truth is defined relative to a structure, a \emph{model}, specified independently of any proof system. 
The divergence between syntactic consequence and semantic consequence becomes a precise theorem rather than a philosophical 
observation. G\"odel's incompleteness theorem(s) reflects that the standard model of arithmetic $\mathbb{N}$
does not admit a recursively enumerable axiomatization.

This orthodox reading came under sustained philosophical scrutiny in
Dummett's 1963 essay ``The Philosophical Significance of G\"{o}del's
Theorem''~\cite{dummett1963}. He argued that the standard interpretation
begs the central philosophical question:
\textit{Are the meaning of mathematical statements given by the conditions under which we are entitled to assert them --- that is, by their proof-conditions --- or by potentially verification-transcendent truth-conditions?}
To assert that $\Con(\pa)$ is true but
unprovable \emph{presupposes} a realist conception of arithmetic, such as $\mathbb{N}$, but it is precisely this presupposition
that is at issue. 

But what other reading is there? Dummett advances his own anti-realist account according to which the content of a mathematical statement is given by the conditions under which it can be recognised as correct. Here no automatic warrant exists for attributing to $\Con(\pa)$ a truth-value that transcends derivability. The real significance of G\"{o}del's theorem(s), in this reading, lies not in vindicating realism but in revealing that the concept of a \emph{ground for asserting something of all the natural numbers} is \emph{indefinitely extensible}: for any definite characterisation of such grounds --- that is, any recursively axiomatised formal system --- the G\"{o}delian construction yields a natural extension of it. This indefinite extensibility is, for Dummett, a form of inherent vagueness in the concept of natural number itself, and as such is perfectly compatible with the thesis that mathematical meaning is constituted by use rather than by reference to a determinate, mind-independent domain.

This proposal attracted both admiration and criticism, and prompted important clarifications from Dummett himself. Wright~\cite{wright1994} subjected its central presupposition to sustained scrutiny. Dummett's account requires that we can always produce an informal demonstration of the G\"{o}del sentence --- that recognising the truth of $\Con(\pa)$ from outside the system is always available to us. Wright argues that this presupposition faces a dilemma. What G\"{o}del's construction strictly delivers is a conditional: \textit{if} the system is consistent, \textit{then} $\Con(\pa)$ is unprovable within it; a separate consistency proof is needed to detach the consequent.  Wright also notes the risk of a simpler conflation underlying the intuitive pull of Dummett's picture: we may be confusing the \emph{discovery of a commitment} --- the recognition that consistency obliges acceptance of $\Con(\pa)$ --- with the independent demonstration of its truth.

Moore~\cite{moore1998}, broadly sympathetic to Dummett's anti-realist orientation, pushed back from a different direction. The threat that G\"{o}del's theorem poses to the thesis that meaning is use is less distinctive than Dummett supposes, since the inability to formally capture the extension of a concept by a recursive axiomatisation is a pervasive feature of natural language, not a peculiarity of arithmetic. Raatikainen~\cite{raatikainen2005} reinforced the epistemological pressure from Wright's direction, arguing that the claim that we can simply \emph{see} the G\"{o}del sentence to be true is considerably more fragile than the anti-mechanist tradition acknowledges, and that the epistemic situation is systematically obscured by the conflation Wright identifies.

These critical exchanges leave Dummett's programme in a philosophically important but technically underdeveloped state. He identified the right question---whether arithmetical correctness can be constitutively determined by inferential practice rather than by correspondence to an independently given structure---but lacked the formal framework to give a precise and rigorous answer. In the decades following his critical essay, Dummett developed his anti-realist, justificatory theory of meaning. We may call this \emph{inferentialism} (in contrast to \emph{denotationalism}); it has received mathematical articulation as \emph{proof-theoretic semantics} (in contrast to \emph{model-theoretic semantics}). 

In his paper ``Classical Logic without
Bivalence'', Sandqvist~\cite{Sandqvist2005inferentialist,Sandqvist_2009} offers a proof-theoretic semantics for classical logic. It proceeds by a
semantic judgment relation called \emph{support} ($\supp$), defined
compositionally from a base of atomic inferential commitments. The
aim of the present paper is to deploy this semantics in the arithmetical
setting, and thereby to give a technically grounded analysis of Dummett's
reading of G\"{o}del.

We show that, for any arithmetical theory $\pa$ satisfying some basic properties --- met by Robinson's $\mathsf{Q}$ and Peano
Arithmetic ($\mathsf{PA}$) --- validates its own consistency statement $\Con(\pa)$; that is,
\[
  A \Vdash \Con(\pa).
\]
This stands in apparent tension with G\"{o}del's second incompleteness
theorem, which asserts $\nvdash_{\pa} \Con(\pa)$. The resolution lies in the scope of Sandqvist's completeness theorem: it holds only
when the underlying signature contains a sufficient reserve of unused
constants, a condition that fails in the standard finite signature of
arithmetic. In the absence of completeness, derivability and support come
apart in a principled way, and it is precisely this gap that the present
result inhabits. Full details and analysis are given in what follows.

The philosophical moral is this: G\"{o}del's theorem need not be read as
demonstrating that a theory of arithmetic falls short of some antecedently
given model. It can instead be understood as a manifestation of the gap between what can be established by
finite derivation and what is determined by the inferential commitments
constitutive of the theory itself. This reframing establishes a principled
separation between two notions of consequence associated with a single theory:
derivability, which operates syntactically within the proof system, and
semantic support, which is fixed by the inferential roles the theory's axioms
assign to its non-logical vocabulary.

\subsection*{Organization.}
In Section~2 we review the proof-theoretic semantic framework and recall Sandqvist's notion of support, giving the compositional clauses for the logical constants and stating the relevant soundness and completeness result. In Section~3 we specify the assumptions on the arithmetical theory $\pa$ with which we work, together with the properties required of the provability predicate. In Section~4 we show that the support relation validates the consistency statement for suitable theories $\pa$, and we prove the main separation result $\pa \nvdash \Con(\pa)$ but $\pa \supp \Con(\pa)$. In Section~5 we explain in detail why this does not conflict with G\"odel's Second Incompleteness Theorem and clarify the precise relationship between semantic support, derivability, and familiar reflection principles.

\section{Proof-theoretic Semantics} \label{sec:pts}

The guiding idea of proof-theoretic semantics is that the meanings of expressions, and in particular of the logical constants, are to be explained in terms of their roles in inference rather than in terms of truth in independently given structures; see, for example, Dummett~\cite{dummett1993theoryI,dummett1991logicalbasis}, Prawitz~\cite{prawitz2006natural,prawitz1973towards}, and Schroeder-Heister~\cite{Schroeder-Heister_2008,Schroeder-Heister_2024}. This idea can be motivated by familiar cases of concept acquisition in natural language. Someone who does not yet know what a \emph{vixen} is may nevertheless come to understand the term by being told, for instance, that an individual is a vixen if and only if she is both a fox and female. More precisely, mastery of the concept is exhibited by grasping the inferential connections that license one to pass from claims about something's being a fox and being female to the claim that it is a vixen, and conversely:
\[
\infer{\mathrm{Vixen}(t)}{\mathrm{Fox}(t) & \mathrm{Female}(t)}
\qquad
\infer{\mathrm{Fox}(t)}{\mathrm{Vixen}(t)}
\qquad
\infer{\mathrm{Female}(t)}{\mathrm{Vixen}(t)}
\]

What is learned in such a case is not an independent condition of satisfaction, but a pattern of correct inferential use. In this sense, meaning is tied to patterns of inference rather than to a prior grasp of a domain of objects.
The proof-theoretic semantic programme applies this idea to logical vocabulary. For example, the meaning of conjunction ($\land$) is given not by a truth-table but by the inferential principles governing its use:
\[
\infer{\varphi \land \psi}{\varphi & \psi}
\qquad
\infer{\varphi}{\varphi \land \psi}
\qquad
\infer{\psi}{\varphi \land \psi}
\]

From this perspective, one is led  to a semantic notion of consequence formulated in proof-theoretic terms. In what follows, we recall the framework introduced by Sandqvist~\cite{Sandqvist2005inferentialist,Sandqvist_2009}, and fix notation and basic results that will be used throughout the paper.

\subsection{Bases and atomic derivability}

In the model-theoretic setting, one typically begins with a mathematical structure $\mathfrak{M}$ intended to represent a mind-independent domain of objects, and defines semantic notions by reference to what is true in that structure. In the present proof-theoretic setting, by contrast, we begin not with a domain but with a body of inferential commitments, understood as constitutive of the agent's use of the non-logical vocabulary. 

This body of commitments is called a \emph{base} $\base{B}$. In keeping with the constructive and syntactic character of the approach, it is represented by a set of inference rules over \emph{atomic} formulas, that is, formulas of the form $P(t_1,\ldots,t_n)$ where $P$ is a predicate symbol and the $t_i$ are terms. In particular, no logical constants occur at this stage: $\bot$ is not treated as an atomic formula, and neither $\to$, $\land$, nor $\forall$ is presupposed. The meanings of the logical constants will instead be fixed later by the compositional semantic clauses that govern their inferential roles.

\begin{definition}[Atomic rule]
An \emph{atomic rule} is an inference figure of the form
\[
\infer{C}{P_1 \quad \ldots \quad P_n}
\]
where $C, P_1, \ldots, P_n$ are closed atomic formulas.
\end{definition}

\begin{definition}[Base]
A \emph{base} $\base{B}$ is a set of atomic rules.
\end{definition}

\begin{example}\label{ex:socrates}
Let $H(t)$ mean ``$t$ is human'', $M(t)$ mean ``$t$ is mortal'', and let $s$ denote Socrates. Suppose an agent is committed to the claims ``Socrates is human'' and ``All humans are mortal''. This is represented by the rules
\[
\infer{H(s)}{} \qquad\text{and}\qquad \infer{M(t)}{H(t)}.
\]
\end{example}

Given a base $\base{B}$ encoding an agent's atomic inferential commitments, there is an induced notion of what the agent is \emph{committed to} in virtue of those commitments. Intuitively, a claim counts as supported at the atomic level if it can be obtained by finitely and iteratively applying the rules in the base. This leads to a natural notion of derivability from a base, defined by closing the atomic rules under their own application.

\begin{definition}[Derivation in a base]\label{def:der-base}
Let $\base{B}$ be a base. The set of $\base{B}$-derivations is the smallest set of trees defined as follows:
\begin{itemize}
  \item[\textsc{Ref}] If $\infer{C}{}$ is in $\base{B}$, then the one-node tree with root $C$ is a $\base{B}$-derivation.
  \item[\textsc{App}] If $\infer{C}{P_1 \ \ldots \ P_n}$ is in $\base{B}$, and $\mathcal{D}_1, \ldots, \mathcal{D}_n$ are $\base{B}$-derivations of $P_1, \ldots, P_n$, then the tree with root $C$ and subtrees $\mathcal{D}_1, \ldots, \mathcal{D}_n$ is a $\base{B}$-derivation.
\end{itemize}
We write $\proves_{\base{B}} C$ if there exists a $\base{B}$-derivation with root $C$.
\end{definition}

\begin{example}[Example~\ref{ex:socrates} cont'd]
From the base in Example~\ref{ex:socrates} we obtain the derivation
\[
\infer{M(s)}{H(s)}
\]
showing that $\proves_{\base{B}} M(s)$.
\end{example}

This notion of derivability $\vdash_{\base{B}}$ operates entirely at the atomic level: it concerns only atomic formulas and the inferential commitments encoded in the base. One may think of a base $\base{B}$ as a very simple axiom system, consisting of rules of the form $P_1, \ldots, P_n \to C$, closed under a single mode of composition, namely \emph{modus ponens}. At this stage, no logical constants are involved and no structure is assumed beyond what is encoded in the atomic rules themselves. The treatment of complex formulas and logical vocabulary will be introduced next, by extending this atomic notion of derivability to a fully compositional proof-theoretic semantics.

\subsection{The support relation}

We now define the central semantic notion of the framework: when a formula is \emph{supported} by a base. This relation plays the role, in the present proof-theoretic setting, that the satisfaction relation plays in model-theoretic semantics.

We write $\setVar, \setTerm, \setAtoms,$ and $\setFormulas$ for the sets of \emph{variables}, \emph{terms}, \emph{atomic formulas}, and \emph{formulas}, respectively. We use $\setClosedTerm$, $\setClosedAtom$, and $\setClosedFormulas$ for the corresponding sets of closed expressions. The logical constants $\to$, $\land$, $\forall$, and $\bot$ are primitive; other connectives and quantifiers are introduced as standard abbreviations: 
\[
\neg \varphi := \varphi \to \bot, \qquad \varphi \lor \psi := \neg(\neg \varphi \land \neg \psi), \qquad \text{and} \qquad \exists x\,\varphi := \neg \forall x\,\neg \varphi
\]

\begin{definition}[Support]\label{def:supp}
The support relation is defined inductively by the clauses in Figure~\ref{fig:support}, where $\base{B}$ and $\base{C}$ are bases, all formulas are assumed to be closed, and $\Delta$ ranges over non-empty sets of closed formulas.
\end{definition}

The general motivation and detailed justification for these clauses will not be repeated here; we defer to Sandqvist~\cite{Sandqvist2005inferentialist,Sandqvist_2009} for details. We note only a few orienting remarks. The clause for $\bot$ goes back to Dummett~\cite{dummett1991logicalbasis} and expresses the idea that a state of commitments is inconsistent precisely when it supports everything. 

The clause (Inf) has the familiar persistence or monotonicity form characteristic of Kripke-style semantics for intuitionistic logic: support is required to be stable under all admissible extensions of the underlying base. This is entirely natural given the proof-theoretic and constructive motivations of the framework. Nevertheless, as the soundness and completeness theorem recalled below shows, the resulting notion of support captures classical logic.

\begin{figure}
\hrule
\begin{align*}
&\supp_{\base{B}} A &&\text{iff } \qquad \supp_{\base{B}} A \tag{At}\\
&\supp_{\base{B}} \varphi \land \psi &&\text{iff } \qquad \supp_{\base{B}} \varphi \text{ and } \supp_{\base{B}} \psi \tag{$\land$}\\
&\supp_{\base{B}} \varphi \to \psi &&\text{iff } \qquad \varphi \supp_{\base{B}} \psi \tag{$\to$}\\
&\supp_{\base{B}} \forall x\varphi &&\text{iff } \qquad \supp_{\base{B}} \varphi[x \mapsto t] \text{ for all closed terms } t \tag{$\forall$}\\
&\supp_{\base{B}} \bot &&\text{iff } \qquad \supp_{\base{B}} A \text{ for every closed atom } A \tag{$\bot$}\\
\Delta &\supp_{\base{B}} \varphi &&\text{iff } \qquad \text{for all } C \supseteq B, \text{ if } \supp_{\base{C}} \psi \text{ for all } \psi \in \Delta, \text{ then } \supp_{\base{C}} \varphi \tag{Inf}
\end{align*}
\hrule
\caption{Support in a Base}
\label{fig:support}
\end{figure}

We now want to turn to defining \textit{validity} ($\supp$). 
As bases model the inferential commitments of an 
idealised rational agents, there are at least two natural candidates. One may quantify over arbitrary 
collections of inferential commitments, assigning no base a privileged 
status:
\[
    \Gamma \supp \varphi \qquad \mbox{iff} \qquad \Gamma \supp_{\base{B}} \varphi \quad \text{for any base } \base{B};
\]
--- this recalls Tarski's definition of model-theoretic consequence. Or one may simply require that $\varphi$ be supported in the absence of 
any antecedent commitments:
\[
   \Gamma \supp \varphi \qquad \mbox{iff} \qquad \Gamma \supp_{\emptyset} \varphi.
\]
These two formulations coincide as shown by Sandqvist~\cite{Sandqvist_2009}. We choose the latter one because it is simpler to state.

\begin{definition}[Validity]
For an arbitrary set $\Gamma \subseteq \setClosedFormulas$,
\[
    \Gamma \supp \varphi 
    \quad \text{iff} \Gamma \supp_{\emptyset} \varphi.
\]
\end{definition}

To recover classical logic, we require a restriction on $\Gamma$ --- there must be an infinite reserve of atomic formulae. Mathematically, this caused by the non-compactness of the semantics. Suppose, for example, that $\Omega$ contained all atoms; it is easy to see that 
\[
\Omega \supp \bot
\]
To block such pathological cases, we put some restrictions on the syntax.

Sandqvist~\cite{Sandqvist2005inferentialist,Sandqvist_2009} and Gheorghiu~\cite{gheorghiu2025fol} resolved this problem by working only with finite contexts. This essentially establishes only weak completeness:
\[
     \supp \varphi 
    \quad \text{iff} \quad \vdash_\Gamma \varphi
\]

However, the technique employed by Gheorghiu~\cite{gheorghiu2025fol} can be generalized. The finiteness condition is required only to guarantee that there are enough constants in the language. Making this the condition explicit and the finiteness condition itself is no longer required. 

\begin{theorem}\label{thm:snc}
Assuming a language with sufficiently many constants,
\[
\Gamma \supp \varphi \quad \text{iff} \quad  \vdash_\Gamma \varphi
\]
\end{theorem}

The soundness direction is easy to see by the standard induction argument. Indeed, it holds locally:

\begin{proposition} \label{prop:local-soundness}
Let $\supp_{\base{B}} \gamma$ for $\gamma \in \Gamma$. If $\vdash_{\Gamma} \phi$, then $\supp_{\base{B}} \phi$
\end{proposition}

The completeness direction is where the condition  `sufficiently many constants' is required. What this means is that the language contain infinitely many atomic expressions that are not used in the theory $\Gamma$. This reflects the ordinary mathematical practice of introduces fresh symbols or parameters as needed. Formally, this amounts to distinguishing between an \emph{ambient-language} signature $\sigma$, which contains an unlimited supply of constants, and the more austere \emph{object-language} sub-signature $\tau$ in which $\Gamma$ is formulated and its proofs are carried out.

For example, one may take the ambient language to have signature $\sigma: = \langle (c_i)_{i\in \mathbb{N}}, S, +, \times \rangle$ --- that is, one has countably many constants $c_i$.  Meanwhile, a theory of arithmetic $\pa$ in which one is interested is specified in the sub-signature $\tau = \langle c_0, S, +, \times \rangle$ with only one constant $c_0=0$. The additional constants serve merely as a stock of fresh names, playing the role of parameters or eigenvariables in semantic and metatheoretic arguments. 

In such an enriched setting, the soundness and completeness theorem applies without difficulty, and in particular one recovers the equivalence
\[
\mathsf{PA} \supp \varphi \quad \text{iff} \quad \vdash_{\mathsf{PA}} \varphi.
\]
when $\mathsf{PA}$ is Peano Arithmetic. 
Nothing essential about ordinary arithmetic is changed by this manoeuvre. Yet is certainly not what is traditionally done in setting up a formal theory of arithmetic. 

If one insists on working strictly within the original, austere signature of arithmetic, and refuses to make use of any such background reserve of symbols, the situation becomes more delicate and, from the present point of view, more interesting. The failure of the infinite reserve condition in this setting opens up the possibility that derivability and semantic support may come apart in a principled way. It is in this gap that G\"odel's incompleteness theorems lie.

\section{Theories of Arithmetic} \label{sec:theories}

In this section we fix the arithmetical background and establish some standard proof-theoretic infrastructure that will be used remainder of the paper. We work with a fixed theory $\pa$ in a first-order signature  $\langle 0, S, +, \times \rangle$, assumed to be strong enough to support the usual arithmetization of syntax and proofs. Nothing here goes beyond completely orthodox accounts; the purpose is simply to make precise, and to fix notation for, the machinery that will later be invoked in the semantic arguments.

We begin with some elementary notational conventions. We write $t<s$ to abbreviate $\exists x(t+x = s)$. Accordingly, we write $\forall x<t\,\varphi$ to abbreviate $\forall x\,(\exists y(x+y=t) \to \varphi)$. We also write $\bar{n}$ for the $n$th numeral --- that is, 
\[
\begin{cases}
    \bar{0}=0 \\
    \bar{n} = S(\overline{n-1}).
\end{cases}
\]

We now turn to the formalization of syntax and proofs inside arithmetic. This is the sense in which $\pa$ is assumed to be \emph{strong enough} to be of interest: for the incompleteness arguments and the later semantic constructions, we require a fully explicit internal notion of proof.

In particular, $\pa$ must be strong enough to define a provability predicate
\[
\Prf_{\pa}(p,x),
\]
intended to express that the natural number $p$ codes a $\pa$--derivation whose final formula has G\"odel number $x$. To this end, we assume there is a G\"odel numbering
\[
\ulcorner\cdot\urcorner : \{\text{finite strings of symbols}\} \to \mathbb{N}
\]
satisfying the usual effectiveness requirements. 

As the setup of provability predicates is doubtless familiar, we shall be relatively terse and merely list the properties we expect. For example, we assume that the encoding function is injective in the following sense:
\begin{description}
\item[Comp] \namedlabel{desc:completeness-encoding}{\textbf{Comp}} (Completeness of G\"odel numbering.)
For all formulas $\varphi,\psi$,
\[
\proves_{\pa} \ulcorner \varphi \urcorner = \ulcorner \psi \urcorner 
\qquad \text{iff} \qquad  
\varphi = \psi.
\]
\end{description} \medskip

We take it that we can define the following with $\pa$ by bounded ($\Delta_0$-)formulas:
\begin{itemize}
\item $\Form(x)$: $x$ is the G\"odel number of a well--formed formula;
\item $\Seq(p,n)$: $p$ codes a finite sequence of length $n$;
\item $\Elt(p,i,y)$: $y$ is the $i$--th element of the sequence coded by $p$;
\item $\Ax_{\pa}(x)$: $x$ is the G\"odel number of an axiom of $\pa$;
\item $\MP(a,b,c)$: $c$ follows from $\pa$ and $b$ by modus ponens;
\item $\Gen(a,c)$: $c$ is obtained from $\pa$ by universal generalization.
\end{itemize}

The provability predicate is built-up from these function. To this end, first define the predicate expressing that a given line in a coded derivation is locally correct:
\[
\begin{aligned}
\Line(p,i)\;:=\;&
\exists y\,\Big(
\Elt(p,i,y)\ \wedge\ \Form(y)\ \wedge\\
&\Big(
\Ax_{\pa}(y)\ \lor\\
&\qquad
\exists j,k<i\;\exists y_1,y_2\,
\big(
\Elt(p,j,y_1)\ \wedge\
\Elt(p,k,y_2)\ \wedge\
\MP(y_1,y_2,y)
\big)\ \lor\\
&\qquad
\exists j<i\;\exists y_1\,
\big(
\Elt(p,j,y_1)\ \wedge\
\Gen(y_1,y)
\big)
\Big)
\Big).
\end{aligned}
\]
All quantifiers in $\Line(p,i)$ are bounded, hence $\Line(p,i)$ is a $\Delta_0$-formula.

Secondly, we define the proof predicate itself. We set:
\[
\begin{aligned}
\Prf_{\pa}(p,x)\;:=\;&
\Form(x)\ \wedge \; \exists n \;\Big(
\Seq(p,n)\ \wedge\
n>0\ \wedge\\
&\qquad
\Elt(p,n-1,x)\ \wedge \;
\forall i<n\;\Line(p,i)
\Big).
\end{aligned}
\]
All quantifiers are now bounded, so $\Prf_{\pa}(p,x)$ is a $\Delta_0$-formula. Observe $\Prf_{\pa}(p,x)$ expresses that $p$ codes a finite sequence of formulas whose final entry has G\"odel number $x$, and every line of the sequence is locally justified according to the axioms and rules of $\pa$.

Finally, we define the associated provability predicate
\[
\Prov_{\pa}(x)\;:=\;\exists p\,\Prf_{\pa}(p,x),
\]
which is a $\Sigma_1$ formula.
\medskip

We assume that the chosen G\"odel coding satisfies the following standard correctness properties, all of which are primitive recursive and whose arithmetical counterparts are provable in $\pa$. They are stated explicitly only to fix notation and to make later internal reasoning completely transparent.

\begin{description}
\item[Seq] \namedlabel{desc:functionality}{\textbf{Seq}} (Functionality of sequence coding.) 
The predicate $\Elt(p,i,y)$ may be treated as a functional relation in the following sense:
\begin{itemize}
\item for fixed $p$ and $i$, there exists at most one $y$ such that $\Elt(p,i,y)$ holds;
\item if $\Seq(p,n)$ holds, then for each $i<n$ there exists exactly one $y$ such that $\Elt(p,i,y)$ holds.
\end{itemize}

\item[Rule] \namedlabel{desc:correctness}{\textbf{Rule}} (Correctness of syntactic predicates.) 
Externally, the relations $\Ax_{\pa}$, $\MP$, and $\Gen$ represent the intended axiom and inference rules of $\pa$. Internally, $\pa$ proves the corresponding representability conditions, in particular:
\begin{itemize}
\item $\Ax_{\pa}(\overline{n})$ holds iff $n$ is the G\"odel number of an axiom of $\pa$;
\item $\MP(\overline{a},\overline{b},\overline{c})$ holds iff $a=\ulcorner\varphi\urcorner$, $b=\ulcorner\varphi\to\psi\urcorner$, and $c=\ulcorner\psi\urcorner$ for some formulas $\varphi,\psi$;
\item $\Gen(\overline{a},\overline{c})$ holds iff $a=\ulcorner\varphi\urcorner$ and $c=\ulcorner\forall x\,\varphi\urcorner$ for some formula $\varphi$ and variable $x$.
\end{itemize}
We work with a proof system in which generalization has no side conditions, or else such conditions are built into $\Gen$.
\end{description}
\medskip

Finally, we record the basic closure properties of proofs under taking initial segments. This will be needed repeatedly when reasoning inside $\pa$ about its own derivations.

To reason internally about proofs, we require the ability to manipulate prefixes of coded derivations. Let $\Pre(p,k)$ be a primitive recursive function coding the initial segment of length $k$ of the sequence coded by $p$. We assume the following standard properties, all provable in $\pa$:

\begin{description}
\item[Len] \namedlabel{desc:prefix-length}{\textbf{Len}} (Length of prefixes.)
\[
\proves_{\pa} \Seq(\Pre(p,k),k).
\]

\item[Pref] \namedlabel{desc:prefix-closure}{\textbf{Pref}} (Prefix closure of proofs.)
\[
\proves_{\pa} 
\Prf_{\pa}(p,\ulcorner\varphi\urcorner) \ \wedge\ \Elt(p,n,\ulcorner\psi\urcorner)
\ \rightarrow\ 
\Prf_{\pa}(\Pre(p,n+1), \ulcorner\psi\urcorner).
\]
\end{description}
\medskip

These assumptions ensure that $\Prf_{\pa}(p,x)$ and $\Prov_{\pa}(x)$ behave, from the internal point of view of arithmetic, exactly as the intended notions of `$p$ is a code of a $\pa$-proof of a formula with code $x$' and `the formula with code $x$ is provable in $\pa$'. This infrastructure will be used in the next sections both to formulate the relevant consistency statements and to relate syntactic derivability to the proof-theoretic semantic notion of support.

\section{The Consistency Statement} \label{sec:main}

Having defined the provability predicate, we can now define the consistency statement:
\[
\Con(\pa) \;:=\; \neg \Prov_{\pa}(\ulcorner\bot\urcorner).
\]
It is the sentence expressing with $\pa$ that there is no code of a $\pa$--proof of $\bot$. In this section we show that,
\[
    \pa \supp \Con(\pa).
\]
It is important to stress at the outset that this is not a claim about 
derivability. G\"odel's second incompleteness theorem remains entirely 
intact: $\pa$ does not \emph{prove} its own consistency,
\[
\not \proves_{\pa} \Con(\pa)
\]
This is as we are exploiting the condition of completeness (Theorem~\ref{thm:snc}) for there to be a sufficient reserve of constants. 

Before the mathematical treatment below, we might offer some explanation of why this is as we would expect.  As mentioned in Section~\ref{sec:introduction}, Dummett claimed that G\"{o}del's incompleteness theorem(s) shows the concept of `ground for asserting
something about all natural numbers' to be \emph{indefinitely extensible}~\cite{dummett1963}. A concept is `indefinitely
extensible' when there is a principle such that, for any definite totality
of things falling under it, that principle generates a strictly more inclusive
definite totality of things falling under it. In this case, any definite characterisation of such grounds for an assertion of about the natural numbers can be
extended by the G\"{o}delian construction to a strictly more inclusive one.

The problem with Dummett's argument is that this extension was supposed to be
one we can always \emph{recognise as correct} --- which requires demonstrating
the truth of the G\"{o}del sentence, which in turn requires a consistency
proof. Wright's central objection turns on a dilemma concerning exactly that
consistency proof~\cite{wright1994}. If it is required to be \emph{suasive} --- genuinely
informative, providing positive grounds for confidence in the system's
soundness --- then no such proof is available in general, and Dummett's
argument fails to go through. If, on the other hand, only a non-suasive proof
is required (one that merely draws out what is already implicit in treating
the system's axioms as true), then the demonstrability of the G\"{o}del
sentence can do no work in establishing that arithmetical truth outruns any
systematic characterisation. Either way, the presupposition on which
Dummett's problem rests is not securely in place.

Dummett conceded the essential force of Wright's objection to the original
formulation~\cite{Dummett1994}. He clarified that the principle of extension associated with an
indefinitely extensible concept need not itself resist precise
characterisation --- on the contrary, it is precisely because the
G\"{o}delian extension procedure \emph{can} be given a general
characterisation that the meaning-is-use thesis is safeguarded. What cannot
be done is to capture the extension of the concept by any fixed recursive
axiomatisation. He also clarified the notion of a \emph{definite totality}: one over
which quantification always yields a statement determinately true or false,
and of which we have a clear grasp --- acknowledging that this standard admits
of degrees, and that whether \emph{natural number} counts as indefinitely
extensible depends on how rigorously one sets those standards, while insisting
that \emph{ordinal number} is indefinitely extensible on any non-trivial
setting, on pain of contradiction.

We pickup on Dummett's thought and attempt to give it mathematical formulation using the proof-theoretic semantics of classical logic. To this end we require some technical details. 

\subsection*{Definitionality of Arithmetic Theories}

On the inferentialist reading, the domain of numbers is not an
independently given totality, but is fixed by the signature and the
inferential rules of the theory itself. Once the meanings of $0$, $S$, $+$,
and $\times$ are determined by the rules of $\pa$, there is nothing further
for a `number' to be beyond what is already fixed by those rules. Importantly, there may well be facts that outrun our ability to \emph{derive} them within the
finitary proof system.

From this point of view, extending $\pa$ by the assertion that there exists a
number coding a $\pa$-proof of $\bot$ is not merely to add a new hypothesis
about an already fixed domain. It is, rather, to \emph{extend} the very
concept of `number' to one in which such a claim is true. However, that renders the concept \emph{ipso facto} incoherent for it claims of itself to be contradictory; therefore,
\[
\pa, \Con(\pa) \supp \bot
\]

This line of thought points to a more general phenomenon, namely a form of
\emph{reflection}:
\[
\pa, \Prov(\ulcorner \varphi \urcorner) \supp \varphi
\]
Let $\varphi$ be any sentence. If the provability predicate is correct, then whenever one extends $\pa$ by the
claim that there exists a number coding a $\pa$-proof of $\varphi$, it must be that $\varphi$ itself
obtains. Using the clause for $\to$, we identify this with the claim:
\[
    \pa \supp \Prov_{\pa}(\ulcorner \varphi \urcorner) \to \varphi.
\]
The consistency statement is simply the special case in which $\varphi$ is
$\bot$.

Here the connection to Dummett's notion of indefinite extensibility becomes
precise. The reflection principle does not merely instantiate that notion ---
it \emph{is} the general characterisation of the principle of extension that
Dummett insisted must be available if the meaning-is-use thesis is to be
safeguarded. Dummett clarified, against Wright, that the principle
of extension for an indefinitely extensible concept can itself be given a
precise characterisation~\cite{Dummett1994}; what resists capture is not the principle but the
totality it generates. 

The reflection schema is exactly such a principle. The
premise $\Prov_{\pa}(\ulcorner \varphi \urcorner)$ does not describe a number
already present in a fixed domain; asserting it is the act of extension.
Each application of the schema enlarges the totality of numbers, and the
G\"{o}delian iteration --- the fact that the procedure can always be applied
again to the resulting system --- is precisely what makes the concept of
number indefinitely extensible in Dummett's sense, and what ensures that no
fixed recursive axiomatisation can ever capture it. 

For this argument to go through mathematically, we use the fact that $\pa$ is strong enough to do basic arithmetic. A sufficient condition is that it decides atoms --- for any closed terms $t_1$ and $t_2$, either
\[
\pa \proves t_1 = t_2 \qquad \text{ or } \qquad  \pa \proves t_1 \neq t_2
\]
--- and can compute using primitive recursion.  This condition is met by both Robinson Arithmetic $\mathsf{Q}$ and Peano Arithmetic $\mathsf{PA}$ --- see Kaye~\cite{Kaye1991}.

\begin{definition}[Sufficiently Strong Arithmetic Theory]
An arithmetic theory $\pa$ is \emph{sufficiently strong} if it satisfies the following two conditions:
\begin{itemize}
    \item For any $m, n \in \mathbb{N}$: $\vdash_{\mathsf{T}} \bar{m} = \bar{n}$ iff $m=n$,
    \item For every $\Delta_0$-sentence $\phi$, either $\vdash_{\pa} \phi$ or $\vdash_{\pa} \neg\phi$.
\end{itemize}
\end{definition}

Importantly, if $\pa$ is sufficient and recursively axiomatizable --- that is, 
there is a primitive recursive function enumerating the axioms of $\pa$ --- then 
$\pa$ can define a provability predicate, as in Section~\ref{sec:theories}.

\subsection*{Maxiconsistent bases}

Intuitively, a base $\base{B}$ represents an information state of an 
idealised reasoner: it records which inferential commitments are 
currently in place, but need not in general be complete. When $\base{B}$ 
supports all axioms of $\pa$ we take this to mean that an agent has 
accepted the laws of arithmetic as constitutive of the subject matter. Nonetheless, an information state $\base{B}$ may be extended by adding 
further commitments, provided this does not lead to incoherence. Bases 
are, in general, partial and extensible. The limiting case is reached 
when no further extension is possible without collapse into triviality.

This motivates the notion of a \emph{maxiconsistent base}, due to 
Makinson~\cite{Makinson_2014}, who argues that maxiconsistent bases 
play a role in the present setting analogous to that played by models 
in classical model theory. 

\begin{definition}[Maxiconsistent Base]
A base $\base{B}$ is \emph{maxiconsistent} if it is consistent 
(i.e.,~$\not\supp_{\base{B}} \bot$) and maximal among consistent 
bases (i.e.,~for every $\base{C} \supseteq \base{B}$, either 
$\base{C} = \base{B}$ or $\base{C}$ is inconsistent).
\end{definition}

Makinson observes that they behave classically with respect to the defined 
connectives:

\begin{proposition}\label{prop:maxiconsistent-disjunction}
If $\mbase{B}$ is maxiconsistent, then:
\begin{itemize}
    \item $\supp_{\mbase{B}} \neg \varphi$ iff $\not\supp_{\mbase{B}} \varphi$;
    \item $\supp_{\mbase{B}} (\varphi \lor \psi)$ iff $\supp_{\mbase{B}} \varphi$
        or $\supp_{\mbase{B}} \psi$;
    \item $\supp_{\mbase{B}} \exists x\,\varphi$ iff $\supp_{\mbase{B}}
        \varphi[x \mapsto t]$ for some $t \in \setClosedTerm$.
\end{itemize}
\end{proposition}

\begin{proof}
We first establish the auxiliary equivalence
\begin{equation}\label{eq:ast}
    \supp_{\mbase{B}} \varphi
    \quad\mbox{iff}\quad
    \not\supp_{\mbase{B}} \neg\varphi.
    \tag{$*$}
\end{equation}
For the left-to-right direction, support of $\varphi$ together with
consistency of $\mbase{B}$ immediately precludes support of
$\neg\varphi$.  For the converse, if $\not\supp_{\mbase{B}} \neg\varphi$
then by the clause for implication, there exists a consistent base
$\base{C} \supseteq \mbase{B}$ with $\supp_{\base{C}} \varphi$;
maximality of $\mbase{B}$ then yields $\supp_{\mbase{B}} \varphi$.
We now verify each clause in turn.

\medskip\textbf{Negation.}\;
Since $\mbase{B}$ is consistent, $\supp_{\mbase{B}} \varphi$ and
$\supp_{\mbase{B}} \neg\varphi$ cannot hold simultaneously, so
$\supp_{\mbase{B}} \neg\varphi$ implies $\not\supp_{\mbase{B}} \varphi$.
The converse follows immediately from~\eqref{eq:ast}.

\medskip\textbf{Disjunction.}\;
By definition, $\supp_{\mbase{B}}(\varphi \lor \psi)$ abbreviates
$\supp_{\mbase{B}} \neg(\neg\varphi \land \neg\psi)$.  By the negation
clause just proved, this is equivalent to
$\not\supp_{\mbase{B}}(\neg\varphi \land \neg\psi)$.  The support clause
for conjunction then gives $\not\supp_{\mbase{B}} \neg\varphi$ or
$\not\supp_{\mbase{B}} \neg\psi$, which by~\eqref{eq:ast} is equivalent
to $\supp_{\mbase{B}} \varphi$ or $\supp_{\mbase{B}} \psi$.

\medskip\textbf{Existential quantification.}\;
By definition, $\supp_{\mbase{B}} \exists x\,\varphi$ abbreviates
$\supp_{\mbase{B}} \neg\forall x\,\neg\varphi$.  By the negation clause,
this is equivalent to $\not\supp_{\mbase{B}} \forall x\,\neg\varphi$.
The support clause for universal quantification then yields
$\not\supp_{\mbase{B}} \neg\varphi[x \mapsto t]$ for some
$t \in \setClosedTerm$, and~\eqref{eq:ast} gives
$\supp_{\mbase{B}} \varphi[x \mapsto t]$, as required.
\end{proof}

Makinson also observes that bases may be extended to a maxiconsistent base in a way 
that preserves what the original base does \emph{not} support.

\begin{proposition}\label{prop:maxiconsistent-extension}
If $\not\supp_{\base{B}} \varphi$, then 
there exists a maxiconsistent base $\mbase{B} \supseteq \base{B}$ 
such that $\not\supp_{\mbase{B}} \varphi$.
\end{proposition}

This follows by an argument analogous to Makinson~\cite{Makinson_2014} 
(Lemma~3.4), with an additional case in the induction to handle the 
universal quantifier. 


These results provide the technical background required for this paper. 

\subsection*{Reflection and a corollary}

We now show that an arithmetic $\pa$ supports
reflection, $\Prov_{\pa}(\varphi) \to \varphi$. The argument follows the structure of a
\emph{na\"{i}ve internal soundness proof}: one reasons by induction on
derivations, observing that axioms are supported outright and that the
rules of inference preserve support, so that every formula derivable in
$\pa$ is itself supported.  

This argument is na\"{i}ve as it cannot be carried out \emph{within} $\pa$ itself. It relies on smoothly moving between meta-level and object-level observations. One direction is enabled by local soundness (Proposition~\ref{prop:local-soundness}); for the other we use the following:

\begin{proposition}\label{prop:equationally-complete}
    Let $\pa$ be sufficiently strong $\base{B}$ be a consistent base that supports all
axioms of $\pa$. For any $\Delta_0$-formula $\delta$, if
\[
    \supp_{\base{B}} \delta,
\]
then
\[
    \proves_{\pa} \delta.
\] 
\end{proposition}
\begin{proof}
We show the contrapositive: if $\not \vdash_{\pa} \delta$, then $\not \supp_{\base{B}} \delta$.

Suppose $\pa \not\proves \delta$. As $\pa$ is sufficiently strong and $\delta$ is a $\Delta_0$-sentences,
$\proves_{\pa} \not \delta$. By local soundness (Proposition~\ref{prop:local-soundness}), we
have $\supp_{\base{B}} \neg \delta$. But then $\supp_{\base{B}} \delta$
would render $\base{B}$ inconsistent, contrary to hypothesis. Hence,
$\not\supp_{\base{B}} \delta$. 
\end{proof}

This allows us to lift desirable properties of the provability relation to the semantic level where can use them. 

\begin{theorem}\label{thm:reflection}
    For any $\varphi$, $\pa \supp \Prov_{\pa}(\ulcorner \varphi \urcorner) \to \varphi$.
\end{theorem}
\begin{proof}
    The statement is equivalent to $\pa, \Prov_{\pa}(\ulcorner \varphi \urcorner) \supp \varphi$. Assume, for contradiction, that this is not the case: $\pa, \Prov_{\pa}(\ulcorner \varphi \urcorner) \not \supp \varphi$. Then there is a base $\base{B}$ such that:
    \begin{itemize}
        \item $\supp_{\base{B}} \alpha \in \pa$,
        \item $\supp_{\base{B}}  \Prov_{\pa}(\ulcorner \varphi \urcorner)$, but
        \item $\not \supp_{\base{B}} \varphi$.
    \end{itemize}
    Without loss of generality, by Proposition~\ref{prop:maxiconsistent-extension}, $\base{B}$ is maxiconsistent. We show that from the first two we get $\supp_{\base{B}} \varphi$, contradicting the third. \medskip

By Proposition~\ref{prop:maxiconsistent-disjunction}, there is a term $p$ such that 
\[
\supp_{\base{B}}  \Prf_{\pa}(p,\ulcorner \varphi \urcorner)
\]
Unfolding the definition of $\Prf$,
this abbreviates
\[
\supp_{\base{B}}\exists n\Big(
\Seq(p,n)\wedge n>0 \wedge \Elt(p,n-1,\ulcorner\varphi\urcorner)\wedge \forall i<n\,\Line(p,i)
\Big).
\]
By Proposition~\ref{prop:maxiconsistent-disjunction} again and the clause for $\land$, there is some $n$ such that
$\base{B}$ supports each conjunct:
\begin{enumerate}
\item $\supp_{\base{B}}\Seq(p,\bar n)$,
\item $\supp_{\base{B}}\bar n>0$,
\item $\supp_{\base{B}}\Elt(p,\overline{n-1},\ulcorner\varphi\urcorner)$,
\item $\supp_{\base{B}}\forall i<\bar n\,\Line(p,i)$.
\end{enumerate}

We show $\supp_{\base{B}}\varphi$ by induction on $n$ --- that is, the length of the coded proof $p$.

\smallskip
\noindent\textbf{Base case:} $n=1$. From (3) we obtain that the $0$th line of $\varphi$ is $\ulcorner \varphi \urcorner$,
\[
\supp_{\base{B}} \Elt(p,0,\ulcorner\varphi\urcorner) \qquad \text{and} \qquad \supp_{\base{B}} \Line(p,0)
\]
From Proposition~\ref{prop:maxiconsistent-disjunction} on $\Line(p,0)$ we get that 
\[
\Elt(p,0,t) \qquad \text{ and } \supp_{\base{B}} \Ax(t)
\]
for some term $t$. As $\Ax$ is $\Delta_0$, from Proposition~\ref{prop:equationally-complete}, we have $\vdash_{\pa} \Ax(t)$. 

As $\Elt(p,0,\ulcorner\varphi\urcorner)$ is also $\Delta_0$, from Proposition~\ref{prop:equationally-complete} we also get $\vdash_{\pa} \Elt(p,0,\ulcorner\varphi\urcorner)$. By functionality of sequence encoding (\textbf{Seq}), we have that $t = \ulcorner\varphi\urcorner$. By correctness of the rules (\textbf{Rule}), we obtain $\varphi = \alpha$ for some $\alpha \in \pa$. Hence, as $\base{B}$ supports the axioms of $\pa$, we have $\supp_{\base{B}} \varphi$.

\smallskip
\noindent\textbf{Inductive step:} Assume $n>1$ and, as \emph{induction hypothesis} (IH), that the claim holds for all
proofs of length $k<n$; that is, if
\[
\supp_{\base{B}} \Prf_{\pa}(p', \ulcorner \varphi' \urcorner) \qquad \text{ and } \qquad \supp_{\base{B}} \Seq(p',\bar{k}) \text{ with } k<n, 
\]
then 
\[
\supp_{\base{B}} \psi
\]
From (4) we obtain
$\supp_{\base{B}}\Line(p,\overline{n-1})$. 
Unfolding the definition and using Proposition~\ref{prop:maxiconsistent-disjunction}, there are terms
$i,j$ and $y_1,y_2,y$ such that
\[
\supp_{\base{B}}\Elt(p,\overline{n-1},y)
\]
and one of the following holds:
\begin{itemize}
\item $\supp_{\base{B}}\Ax(y)$; or
\item $\supp_{\base{B}}
\Elt(p,i,y_1)\wedge \Elt(p,j,y_2)\wedge \MP(y_1,y_2,y)$; or
\item $\supp_{\base{B}}
\Elt(p,j,y_1)\wedge \Gen(y_1,y)$.
\end{itemize}
As these are all $\Delta_0$-formulae, by Proposition~\ref{prop:equationally-complete}, we have that
\[
\vdash_{\pa} \Elt(p,\overline{n-1},y)
\]
and one of the following holds:
\begin{itemize}
\item $\vdash_{\pa}\Ax(y)$; or
\item $\vdash_{\pa}\Elt(p,i,y_1)\wedge \Elt(p,j,y_2)\wedge \MP(y_1,y_2,y)$; or
\item $\vdash_{\pa}
\Elt(p,j,y_1)\wedge \Gen(y_1,y)$.
\end{itemize}
From $\vdash_{\pa} \Elt(p,\overline{n-1},y)$ and (3) it follows from (\textbf{Seq}) that $y= \ulcorner \varphi \urcorner$. We now consider each disjunct in turn and show that $\supp_{\base{B}} \phi$ holds: 
\begin{itemize}
    \item We reason exactly as in the base case.
    \item By invertibility of conjunction, have that
\[
\vdash_{\pa} \Elt(p, i, y_1),\qquad \vdash_{\pa}\Elt(p, j, y_2),\qquad \vdash_{\pa}\MP( y_1, y_2,\ulcorner\varphi\urcorner).
\]
By the correctness of the encoding (\textbf{Rule}), there is a formula $\psi$ such that
$ y_1=\ulcorner\psi\urcorner$ and $ y_2=\ulcorner\psi\to\varphi\urcorner$. Let
$p_1:=\Pre(p, i+\bar{1})$ and $p_2:=\Pre(p, j+\bar{1})$; by the prefix closure on proofs (\textbf{Pref}), 
\[
\vdash_{\pa} \Prf(p_1,\ulcorner\psi\urcorner)
\quad\text{and}\quad
\vdash_{\pa} \Prf(p_2,\ulcorner\psi\to\varphi\urcorner)
\]
By local soundness (Proposition~\ref{prop:local-soundness}),
\[
\supp_{\base{B}}  \Prf(p_1,\ulcorner\psi\urcorner)
\quad\text{and}\quad
\supp_{\base{B}} \Prf(p_2,\ulcorner\psi\to\varphi\urcorner)
\]
By the length of prefixes ($\textbf{Len}$), the length of $p_1$ and $p_2$ are less than the length of $p$. Applying IH to each, we obtain
$\supp_{\base{B}}  \psi$ and $\supp_{\base{B}} \psi\to\varphi$. The desired result obtains by \emph{modus ponens}.

\item This is analogous to the previous case. Using \textbf{Rule}, there is some $\psi$ with
$y_1=\ulcorner\psi\urcorner$ such that $\varphi=\forall x \psi$.  From \textbf{Pref} and \textbf{Len}, we obtain a 
proof of $\psi$. Using Proposition~\ref{prop:local-soundness}, we move back to $\base{B}$. The induction hypothesis yields $\supp_{\base{B}}\psi$; and the
generalization yields $\supp_{\base{B}}\varphi$.
\end{itemize}
This completes the induction.
\end{proof}

\begin{corollary}
    $\pa \supp \Con(\pa)$
\end{corollary}

It is important to emphasize that the foregoing result does not conflict with G\"odel's incompleteness theorems. Since the object-language signature 
\[
\sigma=\langle 0,S,+,\cdot\rangle
\]
is finite, the soundness and completeness result for support (Theorem~\ref{thm:snc}) does not apply. In particular, it is the completeness direction that fails. Thus the present separation between derivability and support does not amount to a derivation of $\mathrm{Con}(A)$, nor does it undermine G\"odel's theorem; rather, it reflects a divergence between proof and semantics internal to the theory. In this sense, the phenomenon exhibited here is a semantics--proof gap rather than a truth--theory gap of the familiar model-theoretic kind.

One may ask what becomes of the analysis if we instead work in an infinitary signature
\[
\tau=\langle (c_i)_{i\in\mathbb{N}}, S,+,\cdot\rangle
\]
containing infinitely many constants. In such a setting the conditions for equational completeness become substantially more demanding. In particular, standard arithmetical theories of interest, such as $\mathsf{Q}$ or $\mathsf{PA}$, cease to be decide equalities. In particular Proposition~\ref{prop:equationally-complete} fails for one may construct a consistent base $\base{B}$ that supports all axioms of the theory together with rules
\[
\infer{c_i=0}{}
\]
for each constant $c_i$; see Gheorghiu~\cite{gheorghiu2026arithemtic}. This base remains consistent while supporting equations not derivable in the underlying theory, showing that equational completeness fails. Consequently, the analysis given above no longer applies in this setting.

A natural attempted repair is to enforce equational completeness by stipulating that all additional constants denote zero. Concretely, one may extend $\mathsf{Q}$ or $\mathsf{PA}$ by axioms $c_i=c_0$ for each new constant $c_i$, yielding expanded theories $\mathsf{Q}^\ast$ and $\mathsf{PA}^\ast$. These enriched theories recover equational completeness. However, the soundness and completeness theorem for support then fails for a different reason: the surrounding semantic signature no longer contains a sufficient reserve of unused constants to sustain the completeness argument. In effect, once all constants are fixed by the theory itself, the open-endedness required for the soundness and completeness result is lost.

These observations help to locate precisely where the present analysis applies. The separation between derivability and support arises in the austere, finitely generated signatures in which arithmetic is ordinarily formulated. When the signature is expanded so as to restore completeness of the support semantics, the equational completeness conditions required for the reflection argument cease to hold, and conversely. The phenomenon described here therefore depends essentially on the interaction between a fixed, finite signature and the inferential determination of its terms.

\section{Discussion} \label{sec:discussion}

The central result of this paper establishes that for a sufficiently strong,
primitive recursive arithmetic theory $\pa$ one has
\[
  \not\vdash_{\pa} \Con(\pa) \quad \text{but} \quad \pa \Vdash \Con(\pa).
\]
This establishes a principled separation between two notions of consequence
associated with $\pa$: derivability, which operates deductively within the proof
system, and semantic support, which is fixed by the inferential roles the
theory's axioms assign to its non-logical vocabulary. The purpose of this
section is to assess the conceptual significance of this separation, situate it
within the relevant literature, and indicate its broader implications and
limitations.

We work within Dummett's programme in which the problems of metaphysics are
addressed by establishing a theory of meaning. We take Sandqvist's
proof-theoretic semantics for classical logic as the formal development of such
a theory. In this setup, $\pa$ fixes the \emph{meanings} of the non-logical
vocabulary in the signature --- $0$, $S$, $+$, and $\times$ --- through the roles
these symbols play in inference. Just as the semantics of a logic is not the
same as derivability in that logic, the meaning conferred by $\pa$ on these
symbols is not the same as what can be derived from them.

The separation between derivability and support is not, on this reading, an
artifact of stipulation; it is a clarification of a boundary already latent in
arithmetical practice---the boundary between semantics and derivability.
Goldbach's Conjecture, for instance, is the statement that any even number can
be expressed as the sum of two primes ($\forall x\, G(x)$): for any natural
number $n$, there exist primes $p_1$ and $p_2$ such that $2n = p_1 + p_2$. Its
meaning is given by the totality of instances $G(n)$, ranging over all $n$. If
our mathematical resources are suitably constrained, it may be that no single
proof can deliver the universal generalisation; instead, countably many proofs
may be required to handle the cases of $G(n)$ severally.

The result $\pa \Vdash \Con(\pa)$ alongside $\not\vdash_{\pa} \Con(\pa)$ is
another manifestation of this separation. It asserts that $\Con(\pa)$ is
`true' by the inferential meaning of the non-logical vocabulary as determined
by $\pa$, while not being derivable within it. In the semantic judgment, $\pa$
plays the role that $\mathbb{N}$ plays in model-theoretic consequence:
$\mathbb{N} \models \Con(\pa)$ asserts that $\Con(\pa)$ is `true' by the
denotational meaning of the non-logical vocabulary as determined by
$\mathbb{N}$. The significance of the present approach is that taking $\pa$
itself as determining meaning resists the metaphysical inflation involved in
demanding that $\mathbb{N}$ exists as an independently given structure and that
we have cognitive access to the truths it contains.

The more general result underlying the corollary --- the reflection principle
$\pa \Vdash \Prov(\ulcorner \phi \urcorner) \to \phi$ --- is the main
contribution of this paper. That reflection obtains is, we have argued, a
formal expression of Dummett's notion of arithmetic as indefinitely
extensible~\cite{dummett1963}. As such, it enables a more precise engagement
with the epistemological challenges to Dummett's position, and in particular
with the objection pressed by Wright~\cite{wright1994}.

Wright argued that any informal demonstration of $\Con(\pa)$ depends on prior,
non-circular grounds for consistency, and that absent such grounds one has at
most a conditional --- if $\pa$ is consistent, then $\Con(\pa)$ --- rather than a
free-standing argument. The present result does not answer this challenge
directly, but it allows the challenge to be located with greater precision than
was previously possible.

Wright's dilemma presupposes that any consistency argument must be either
\emph{suasive} --- genuinely informative, providing new grounds for confidence in
$\pa$'s soundness from outside the theory --- or \emph{non-suasive}, merely
drawing out what is already implicit in accepting $\pa$'s axioms, and therefore
epistemically inert. This framing, however, is tailored to a model-theoretic
conception of meaning, in which consistency is a property of a theory evaluated
against something external to it. Once meaning is taken to be constituted by
inferential role rather than by correspondence to an independently given
structure, the dilemma loses its grip: there is no external standard against
which to measure suasiveness, and the distinction between unpacking axioms and
determining semantic consequences comes apart.

The result $\pa \Vdash \Con(\pa)$ occupies neither horn straightforwardly. By
the clause for implication in the support semantics, evaluating a consequence
of the meanings fixed by $\pa$ requires considering extensions of
it. Accordingly, we identify $\pa \Vdash \Con(\pa)$ with the
incoherence of extending $\pa$ by the claim that there exists a proof of
absurdity --- that is,
\[
  \pa \Vdash \Con(\pa) \quad \text{iff} \quad \pa,\, \Prov(\ulcorner \bot
  \urcorner) \Vdash \bot.
\]
This is not a claim imported from outside the theory, nor is it a trivial
unpacking of the axioms. It is rather the semantic consequence of what numbers
\emph{mean} within $\pa$: asserting the existence of a number coding a proof of
contradiction is incoherent given the inferential roles that fix the concept of
number in the first place. Whether this constitutes a genuine third position
that escapes Wright's dilemma, or whether on closer inspection it collapses
into the non-suasive horn, is a question that falls to Wright's camp to
adjudicate. What the present result establishes is that the dilemma cannot
simply be assumed to be exhaustive in the proof-theoretic setting.

What we learn is that $\Con(\pa)$ is `true' according to $\pa$ just in case 
$\Prov(\ulcorner \bot\urcorner)$ is inconsistent with $\pa$. To answer 
Wright's challenge is therefore at least to explain how this inconsistency is 
established, whether suasively or non-suasively. This turns on an analysis of 
the mathematical resources employed in the proof of 
Theorem~\ref{thm:reflection}. We are \emph{not} restricted to the resources 
of $\pa$ itself; such a restriction would be circular, amounting to an attempt 
to show $\vdash_{\pa} \Con(\pa)$. Dummett's programme stems from the thesis 
that language and logic are prior to 
metaphysics~\cite{dummett1991logicalbasis}; accordingly, our resources are not to be limited by a theory of arithmetic but rather by a theory of meaning. 

Dummett's is clear that the neutral position is constructive; that is, we may use constructive reasoning as given by, for example, intuitionistic logic to reason about our theory of meaning. In ``Justification of Deduction'' he calls the justification of classical modes of inference the `one of the most fundamental and intractable problems... in all
philosophy'~\cite{dummett1978justification}. It is precisely this problem that is addressed by Sandqvist's proof-theoretic semantics with which we have been working. Crucially, 
Sandqvist~\cite{Sandqvist2005inferentialist} takes care to establish this 
semantics constructively. As we are consciously adopting this semantics, we are indeed entitled to all of its resources (including classical modes of inference).  

This contrasts instructively with the consistency proof of $\mathsf{PA}$ developed by
Gheorghiu~\cite{gheorghiu2026arithemtic}, which
is a constructive existence proof in the sense Dummett's programme requires:
one constructs explicitly a base $\base{A}$ supporting the axioms of $\pa$
(i.e., $\Vdash_{\base{A}} \alpha$ for each $\alpha \in \mathsf{PA}$) and shows
by induction that it is consistent (i.e., $\not\Vdash_{\base{A}} \bot$). This
does not conflict with G\"{o}del's incompleteness theorems: the argument
assumes that the support relation is already defined, and the definition of
support exceeds what can be expressed or established from within $\mathsf{PA}$
itself. This is precisely as it should be if logic precedes theory; the
semantic framework belongs to the prior logical layer, not to the arithmetical
theory it is used to evaluate.

The analysis of this paper applies specifically to arithmetical theories
formulated in a fixed, finite signature. Within this constraint, classical
derivability is incomplete with respect to the support semantics, and
G\"{o}del's incompleteness theorems emerge as an instance of
semantics-derivability incompleteness rather than truth-theory incompleteness.
Several directions for further work present themselves.

On the technical side, it would be natural to examine whether other undecidable
sentences---including G\"{o}del sentences and instances of uniform
reflection---are supported in the present sense, and to explore the interaction
between support and the ordinal analysis of reflection principles. Since
Kreisel~\cite{Kreisel1965}, reflection principles have been studied as
articulations of what is implicitly accepted when a theory is taken to be
correct. Subsequent work by
Feferman~\cite{Feferman1962,Feferman_1975} developed these ideas through
iterated reflection and autonomous progressions, showing that the
proof-theoretic strength of extensions of $\mathsf{PA}$ obtained by iterating
reflection coincides with that of predicative analysis. In these approaches,
reflection is typically justified by meta-theoretic or epistemic considerations
and corresponds to a genuine strengthening of the underlying theory. The
present framework raises the question of whether such strengthenings can be
recast as the unfolding of semantic content already implicit in the base
theory, rather than as genuinely new commitments.

It would also be of interest to extend the analysis to fragments of
second-order logic~\cite{gheorghiu2025prooftheoreticsemanticssecondorderlogic},
where the concerns raised by V\"{a}\"{a}n\"{a}nen~\cite{vaananen2012second}
about the set-theoretic presuppositions of full second-order logic become
directly relevant.

\bibliographystyle{asl}
\bibliography{bib}

\end{document}

%% file: macros.tex
\usepackage{graphicx} 
\usepackage{amssymb}
\usepackage{proof}
\usepackage{enumitem}

\newtheorem{definition}{Definition}
\newtheorem{example}[definition]{Example}

\newtheorem{corollary}[definition]{Corollary}
\newtheorem{proposition}[definition]{Proposition}
\newtheorem{theorem}[definition]{Theorem}

\newcommand{\setVar}{\mathcal{V}}

\newcommand{\setTerm}{\mathcal{T}}
\newcommand{\closure}{\textsc{cl}}
\newcommand{\setClosedTerm}{\closure(\setTerm)}
\newcommand{\setClosedAtom}
{\closure(\setAtoms)}
\newcommand{\setClosedFormulas}{\closure(\setFormulas)}
\newcommand{\setAtoms}{\mathbb{A}}

\newcommand{\setFormulas}{\mathbb{F}}

\newcommand{\proves}{\vdash}
\newcommand{\supp}{\Vdash}

\newcommand{\base}[1]{\mathfrak{#1}}
\newcommand{\mbase}[1]{{\mathfrak{#1}^\ast}}

\usepackage{color}

\renewcommand{\phi}{\varphi}
\renewcommand{\epsilon}{\varepsilon}
\renewcommand{\emptyset}{\varnothing}

\newcommand{\pa}{\textsf{A}}

\newcommand{\Prf}{\textsf{Prf}}

\newcommand{\Line}{\textsf{Line}}
\newcommand{\Form}{\textsf{Form}}
\newcommand{\Ax}{\textsf{Ax}}

\newcommand{\Seq}{\textsf{Seq}}
\newcommand{\MP}{\textsf{MP}}
\newcommand{\Gen}{\textsf{Gen}}

\newcommand{\Prov}{\textsf{Prov}}
\newcommand{\Con}{\textsf{Con}}
\newcommand{\Elt}{\textsf{Elt}}

\newcommand{\Pre}{\textsf{Pref}}